\documentclass[11pt,reqno]{amsart}

\usepackage[T1]{fontenc}

\usepackage{amsmath}								
\usepackage{amssymb}
\usepackage{amsthm}
\usepackage{amscd}
\usepackage{amsfonts}
\usepackage{stmaryrd}
\usepackage[all]{xy}

\usepackage{caption}
\usepackage{subcaption}
\usepackage{euler}

\usepackage{extarrows}

\usepackage[colorlinks, linktocpage, citecolor = purple, linkcolor = blue]{hyperref}
\usepackage{color}

\usepackage{tikz}									
\usetikzlibrary{matrix}
\usetikzlibrary{patterns}
\usetikzlibrary{matrix}
\usetikzlibrary{positioning}
\usetikzlibrary{decorations.pathmorphing}
\usetikzlibrary{cd}

\usepackage{fullpage}
\usepackage[shortlabels]{enumitem}

\numberwithin{equation}{section}

\newtheorem{maintheorem}{Theorem}	
								
\newtheorem{theorem}{Theorem}[section]
\newtheorem{lemma}[theorem]{Lemma}
\newtheorem{proposition}[theorem]{Proposition}

\theoremstyle{definition}
								
\newtheorem{definition}[theorem]{Definition}

\newtheorem{notation}[theorem]{Notation}

\newtheorem{remark}[theorem]{Remark}

\newcommand{\LL}{\mathbb{L}}

\newcommand{\CC}{\mathbb{C}}

\newcommand{\A}{\mathbb{A}}
\newcommand{\G}{\mathbb{G}}
\newcommand{\GG}{\mathbb{G}}

\newcommand{\calM}{\mathcal{M}}

\newcommand{\calMbar}{\overline{\mathcal{M}}}

\DeclareMathOperator{\Spec}{Spec}

\DeclareMathOperator{\Div}{Div}

\DeclareMathOperator{\calDiv}{\mathcal{D}iv}
\DeclareMathOperator{\calDivbar}{\overline{\mathcal{D}iv}}

\DeclareMathOperator{\Hilb}{Hilb}
\DeclareMathOperator{\zdiv}{Z_{div}}
\DeclareMathOperator{\zmot}{Z_{mot}}

\title[Divisorial motivic zeta functions for marked stable curves]{Divisorial motivic zeta functions for marked stable curves }

\author{Madeline Brandt}
\address{Department of Mathematics, University of California, Berkeley, 970 Evans Hall, Berkeley, CA 94720}
\email{\href{mailto:brandtm@berkeley.edu}{brandtm@berkeley.edu}}

\author{Martin Ulirsch}
\address{Institut f\"ur Mathematik, Goethe-Universit\"at Frankfurt, 60325 Frankfurt am Main, Germany}
\email{\href{mailto:ulirsch@math.uni-frankfurt.de}{ulirsch@math.uni-frankfurt.de}}

\begin{document}

\begin{abstract} We define a divisorial motivic zeta function for stable curves with marked points which agrees with Kapranov's motivic zeta function when the curve is smooth and unmarked. We show that this zeta function is rational, and give a formula in terms of the dual graph of the curve.
\end{abstract}

\maketitle
\setcounter{tocdepth}{1}
\section{Introduction}

Let $k$ be an algebraically closed field. The \emph{Grothendieck ring} $K_0(\text{Var}/k)$ \emph{of varieties over $k$} is the free abelian group on the set of isomorphism classes of varieties modulo the relations
$
{[X] = \big[X \backslash{Y}\big] + [Y]},
$
where $Y$ is a closed subvariety of $X$. It naturally carries a product given by taking the product of varieties:
$
{[X] \cdot [Y] = \big[(X \times_{k} Y)\big]}.
$
For simplicity, we assume that $k$ has characteristic zero. Otherwise, we must instead work with $\widetilde{K}_0(\text{Var} / k)$, which is the quotient of $K_0(\text{Var} / k)$ by the relations generated by $[X] - [Y]$ whenever there is a radical surjective morphism $X\rightarrow Y$ of varieties over $k$; the product in $\widetilde{K}_0(\text{Var} / k)$ is given by the reduced product of algebraic varieties.
Throughout, we denote the class of $\mathbb{A}^1_k$ by $\LL$ in $K_0(\text{Var} / k)$.
 
Let $X$ be a quasiprojective variety over $k$. For $d \geq 1$, the symmetric group $S_d$ acts on $X^d$, and the quotient by this action gives $X_d$, the \emph{d-th symmetric product} of $X$. By convention, we set $X_0 = \text{Spec}\ k$. 
Kapranov \cite{Kapranov} defines the \emph{motivic zeta function of} $X$ with coefficients in the Grothendieck ring:
$$
Z_{\text{mot}}\big(X;t\big) = \sum_{d \geq 0} \big[X_d\big] t^d  \ \ \ \ \ \ \in\ \  1 + t \cdot K_0(\text{Var} / k)\llbracket t\rrbracket.
$$
 This generalizes Weil's zeta function for varieties over finite fields to the motivic setting. When $X$ is a smooth projective curve, $Z_{\text{mot}}(X,t)$ is rational (see e.g. \cite{Kapranov} and \cite{Litt_zetaforcurves}).


In this paper we propose a natural generalization $\zdiv(X,\vec{p})$ (see Definition \ref{def:zdiv}) of Kapranov's motivic zeta function for a stable curve $X$ with $n$ marked points $\vec{p}$ that takes into account the behavior at the nodes and the marked points. The basic idea is to replace the symmetric power $X_d$ in the definition of Kapranov's zeta function by the fiber over $(X,\vec{p})$ in a quotient of Hassett's moduli space of weighted stable curves of type $(1^n,\epsilon^d)$ (as in \cite{Hassett}). In the case $n=0$, this space functions as a natural desingularization of the moduli space of effective divisors on $X$ (see \cite[Section 2]{MoellerUlirschWerner}). 
When $X$ is smooth and does not have marked points,
our coefficients equal the symmetric power, giving $Z_{\text{div}}(X,t) = Z_{\text{mot}}(X,t)$. 
Our main result is the following Theorem~\ref{thm:main}.

\begin{maintheorem}
\label{thm:main}
Let $(X,\vec{p})$ be a stable quasiprojective curve over $k$ with $n$ marked points $\vec{p}$. Then $\zdiv(X, \vec{p};t)$ is rational over $K_0(\text{Var} / k)$. Moreover, if $G = (E,V)$ is the dual graph of $X$, then
$$
\zdiv\big(X,\vec{p};t\big) = \left(\frac{1-\LL t}{1-\LL t-t+t^2}\right)^{|E|+n} (1-t)^{2 |E|+n} \prod_{v \in V}  \zmot\big(\widetilde{X}_v;t\big),
$$
where $\widetilde{X}_v$ is the normalization of the component of $X$ corresponding to the vertex $v \in V$.
\end{maintheorem}

In \cite{BejleriRanganathanVakil} Bejleri, Ranganathan, and Vakil define a motivic Hilbert zeta function $Z_{\Hilb}(X;t)$, where the coefficients are given by Hilbert schemes of points on a variety $X$.
Their zeta function is sensitive to the singularities of $X$, while also agreeing with the usual motivic zeta function when $X$ is smooth.
They show that the motivic Hilbert zeta function of a reduced curve is rational. In contrast to our divisorial zeta function, the motivic Hilbert zeta function in \cite{BejleriRanganathanVakil} does not take into account marked points. Using \cite[Lemma 2.1, Corollary 2.2, and Proposition 6.1]{BejleriRanganathanVakil}, one can calculate that, for a nodal quasiprojective curve $X$ with dual graph $G=(V,E)$, we have 
\begin{equation}\label{eq_Hilbertzeta}
    Z_{\Hilb}\big(X;t\big)=\left({1-t+\LL t^2}\right)^{\vert E\vert}\cdot \prod_{v\in V} Z_{\text{mot}}\big(\widetilde{X}_v;t\big) \ .
\end{equation}

It is instructive to compare our formula in Theorem \ref{thm:main} as well as formula \eqref{eq_Hilbertzeta} for the Hilbert motivic zeta function with the formula for the usual Kapranov motivic zeta function $Z_{\text{mot}}(X;t)$. Using \cite[Chapter 7, Proposition 1.1.7]{ChambertLoirNicaiseSebag} (which is also stated as Lemma \ref{lemma:symmetricpower} below) one may calculate that
\begin{equation*}
    Z_{\text{mot}}\big(X;t\big)=(1-t)^{\vert E\vert}\cdot \prod_{v\in V} Z_{\text{mot}}\big(\widetilde{X}_v;t\big) \ .
\end{equation*}
While $Z_{\text{mot}}\big(X;t\big)$ appears to be insensitive to the nodal singularities of $X$, both $\zdiv\big(X;t\big)$ and $Z_{\Hilb}\big(X;t\big)$ "see" the nodes by adding extra components. 

Kapranov's motivic zeta function is known to be irrational for many surfaces. Let $X$ be a smooth projective connected surface. In \cite{LarsenLunts_rationality} Larsen and Lunts prove that $X$ is only pointwise rational when $X$ has Kodaira dimension $-1$ (over $\CC$) and in \cite{LarsenLunts_motivicmeasures} they show (over any field) that $Z_{\text{mot}}(X;t)$ is not pointwise rational when the Kodaira dimension of $X$ is $\geq 2$. 

\subsection*{Acknowledgements}
We thank Dori Bejleri, who also pointed out an inaccuracy in formula~\eqref{eq_Hilbertzeta}, Dhruv Ranganathan and Bernd Sturmfels for several remarks and interesting questions. Work for this paper has been done while the first author visited University of Warwick and Goethe-Universit\"{a}t Frankfurt. She thanks both universities for their hospitality and providing a stimulating working environment.
This material is based upon work supported by the National Science Foundation Graduate Research Fellowship Program under Grant No. DGE 1752814. Any opinions, findings, and
conclusions or recommendations expressed in this material are those of the authors and do not
necessarily reflect the views of the National Science Foundation.
This project  has  received  funding  from  the  European Union's Horizon 2020 research and innovation programme  under the Marie-Sk\l odowska-Curie Grant Agreement No. 793039. \includegraphics[height=1.7ex]{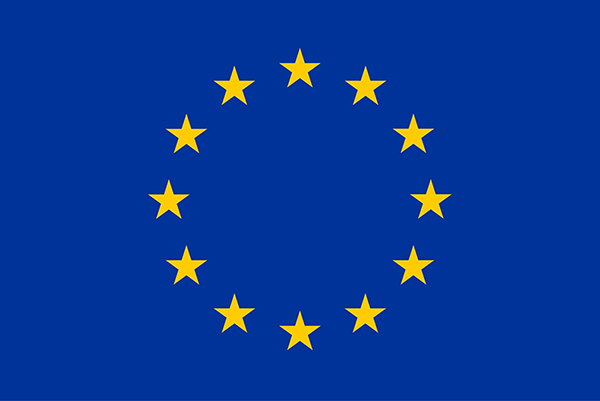} 
 We also acknowledge support from the LOEWE-Schwerpunkt ``Uniformisierte Strukturen in Arithmetik und Geometrie''.

\section{Effective divisors on pointed stable curves}

Let $k$ be an algebraically closed field of characteristic $0$ and let $g,n\geq 0$ such that $2g-2+n>0$.

\begin{definition}
Define a category $\calDivbar_{g,n,d}$ fibered in groupoids over schemes, whose objects are tuples $(\pi\colon X'\rightarrow S,\vec{p}', D)$ consisting of the following data:
\begin{enumerate}[(i)]
\item $\pi\colon X'\rightarrow S$ is a flat and proper morphism of connected nodal curves;
\item $\vec{p}'$ is an ordered collection of sections $p_1',\ldots, p_n'\colon S\rightarrow X$ that do not meet the nodes in each fiber $X'_s$ of $\pi$; and
\item $D$ is a relative effective Cartier divisor of degree $d$ on $X'$ over $S$, whose support does not intersect the nodes and sections in each fiber $X'_s$ of $X'$ over $S$.
\end{enumerate}
We also require that the twisted canonical divisor 
\begin{equation*}
    K_\pi+\epsilon D +p'_1 +\ldots + p'_n
\end{equation*} is $\pi$-relatively ample, where $\epsilon=\frac{1}{d}>0$. 
\end{definition} 

Denote by $\calMbar_{g,1^n,\epsilon^d,}$ the moduli space of weighted stable curves of genus $g$ with $n$ marked points of weight one and $d$ marked points of weight $\epsilon=\frac{1}{d}>0$ in the sense of \cite{Hassett}. There is a natural operation of $S_d$ on $\calMbar_{g,1^n,\epsilon^d}$ that permutes the $d$ marked points of weight $\epsilon$. Then $\calDivbar_{g,n,d}$ is naturally equivalent to the relative coarse moduli space of 
\begin{equation*}
    \big[\calMbar_{g,1^n,\epsilon^d}/S_d\big]
\end{equation*}
over $\calMbar_{g,n}$ in the sense of \cite[Theorem 3.1]{AbramovichOlssonVistoli}. So, in particular, it is a smooth and proper Deligne-Mumford stack with a projective coarse moduli space. There is a natural forgetful morphism $\calDivbar_{g,n,d}\rightarrow \calMbar_{g,n}$ and we write $\calDiv_{g,n,d}$ for its restriction to $\calM_{g,n}$. The complement of $\calDiv_{g,n,d}$ in $\calDivbar_{g,n,d}$ has (stack-theoretically) normal crossings. 

\begin{remark}
For $n=0$, the moduli space $\calDivbar_{g,d}$ was constructed in \cite[Section 2]{MoellerUlirschWerner}. It is also equal to a special case of the moduli space of stable quotients, as defined in \cite[Section~4]{MarianOpreaPandharipande}. 
\end{remark}

Let $(X,\vec{p})=(X,p_1,\ldots, p_n)$ be a stable marked curve of genus $g$ given by a morphism $\Spec(k) \rightarrow \calMbar_{g,n}$. The fiber over this point is given by
$$
\Div^+_d\big(X,\vec{p}\big):=\calDivbar_{g,n,d} \times_{\calMbar_{g,n}} \Spec(k).
$$
This describes tuples $(X',\vec{p}',D)$ consisting of
\begin{enumerate}[(i)]
    \item a nodal curve $X'$;
    \item a collection of marked points $\vec{p}'=(p_1',\ldots,p_n')$ of $X'$ such that $p_1',\ldots,p_n'$ do not meet the nodes of $X'$ and the stabilization of $(X',\vec{p}')$ is isomorphic to $(X,\vec{p})$;
    \item a relative effective Cartier divisor $D$ of degree $d$ on $X'$ whose support does not intersect the nodes or marked points of $X'$.
\end{enumerate}
We also require that the twisted canonical divisor
$$
K + \epsilon D + p_1' + \cdots + p_n'
$$
is ample, where $\epsilon = \frac{1}{d} > 0$. If $X$ is smooth and does not have marked points, the space  $\Div_{d}^+(X)$ gives effective divisors on $X$ and is the $d$-th symmetric power $X_d$ (see \cite[Theorem 3.13]{Milne_Jacobian}).

Suppose now that $X$ is quasiprojective. We choose a compactification $\overline{X}$ of $X$ by smooth points and define $\Div_d^+(X,\vec{p})$ to be the open locus in $\Div_d^+(\overline{X},\vec{p})$ where the support of $D$ does not intersect the preimage of the boundary $\overline{X}-X$ in $X'$. This does not depend on the choice of $\overline{X}$.

Now, we describe the strata of $\Div^+_d(X,\vec{p})$ as in \cite{ourpaper}. We associate to $(X',\vec{p}',D)$ a dual stable pair $\big(G', \text{mdeg}(D)\big)$ as follows: The graph $G'$ is the dual graph of $(X',\vec{p}')$, where the vertices $v$ of $G'$ each correspond to a component $X'_v$ of $X'$. 
For a node between components $X'_v$ and $X'_{v'}$ of $X$, there is an edge between vertices $v$ and $v'$ of $G'$. For a marked point in a component $X_{v}$ we add a leg at $v$.
The restriction of $D$ to each component $X'_v$ defines a divisor $\text{mdeg}(D) = \sum_{v} \deg(D|_{X'_v}) \cdot v$ on $G'$. The graph $G'$ is a subdivision of $G$, the dual graph of $(X, \vec{p})$. The pair $(G',D)$ is a stable pair over $G$, meaning that the degree of $D$ is at least 1 on all exceptional vertices of $G'$. Denote by
$\Delta(G,d)$ the collection of all stable pairs of degree $d$ over $G$.

One can generalize the results in \cite[Section 3.2]{ourpaper} to show that the strata of  $\Div^+_d(X,\vec{p})$ are precisely the locally closed subsets on which the dual pairs are constant.  Denote by  $\Div_{(G',D)}^+(X,\vec{p})$ the locus of points in $\Div^+_d(X,\vec{p})$ whose dual pair is $(G',D)$.

\begin{figure}
\centering
\begin{subfigure}{.25\textwidth}
  \centering
  \includegraphics[height=1.6 in]{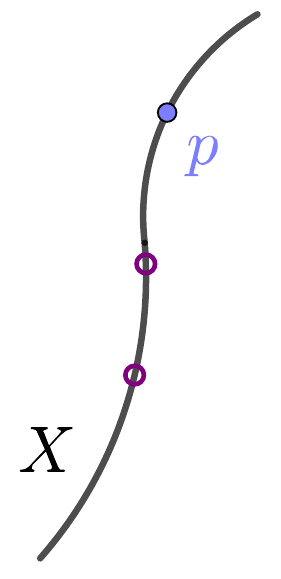}
\end{subfigure}%
\begin{subfigure}{.25\textwidth}
  \centering
  \includegraphics[height=1.6 in]{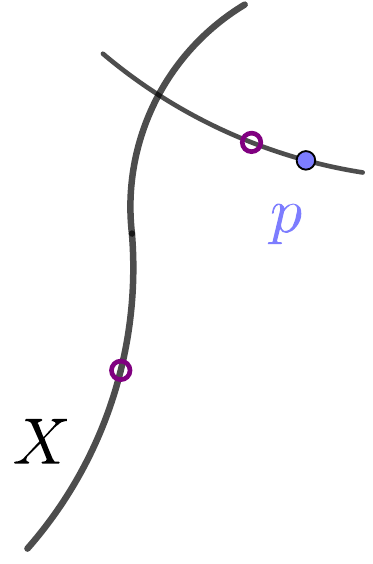}
\end{subfigure}%
\begin{subfigure}{.25\textwidth}
  \centering
  \includegraphics[height=1.6 in]{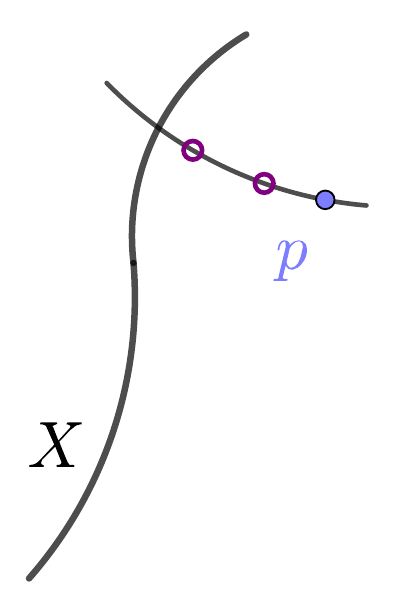}
\end{subfigure}%
\begin{subfigure}{.25\textwidth}
  \centering
  \includegraphics[height=1.6 in]{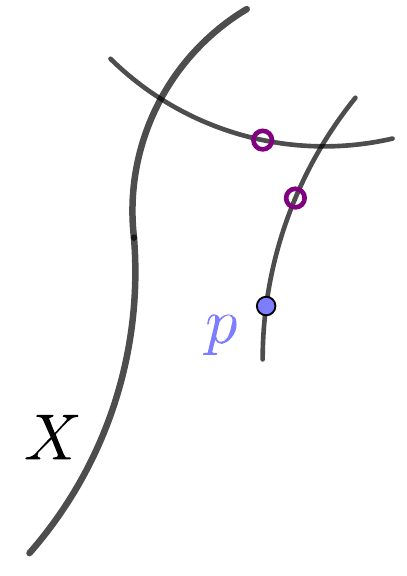}
\end{subfigure}%
\caption{Let $(X,p)$ be a smooth curve with genus $g \geq 1$ and one marked point $p$. In this case, $\Div^+_2(X,p)$ has four strata, corresponding to the pictured combinatorial types of marked stable curves and divisors.}
\label{fig:strata}
\end{figure}

\begin{figure}
\centering
\begin{subfigure}{.142\textwidth}
  \centering
  \includegraphics[height=1.6 in]{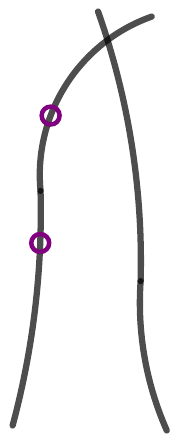}
\end{subfigure}%
\begin{subfigure}{.142\textwidth}
  \centering
  \includegraphics[height=1.6 in]{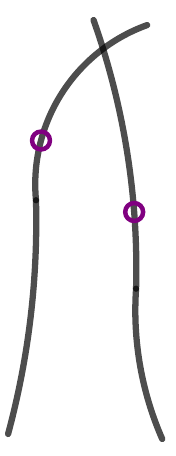}
\end{subfigure}%
\begin{subfigure}{.142\textwidth}
  \centering
  \includegraphics[height=1.6 in]{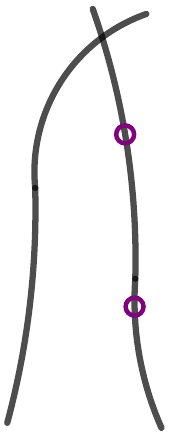}
\end{subfigure}%
\begin{subfigure}{.142\textwidth}
  \centering
  \includegraphics[height=1.6 in]{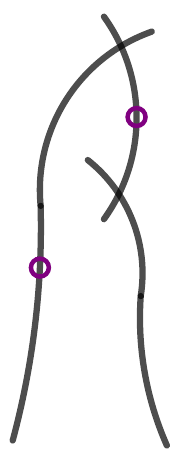}
\end{subfigure}%
\begin{subfigure}{.142\textwidth}
  \centering
  \includegraphics[height=1.6 in]{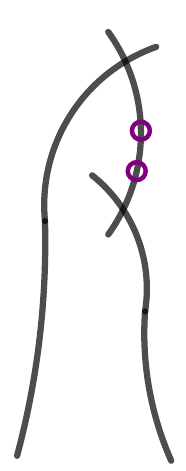}
\end{subfigure}%
\begin{subfigure}{.142\textwidth}
  \centering
  \includegraphics[height=1.6 in]{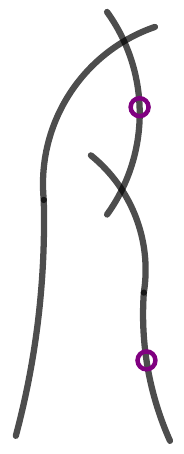}
\end{subfigure}%
\begin{subfigure}{.142\textwidth}
  \centering
  \includegraphics[height=1.6 in]{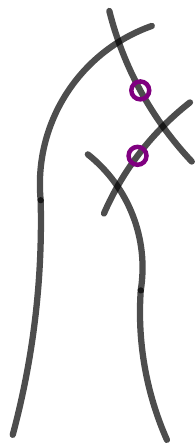}
\end{subfigure}%
\caption{Let $X$ be a curve with two smooth components each having genus larger than one meeting in a node. In this case, $\Div^+_2(X)$ has seven strata, corresponding to the pictured combinatorial types of stable curves and divisors.}
\label{fig:strata2}
\end{figure}

\section{Divisorial motivic zeta function}

\begin{definition}
\label{def:zdiv}
 Let $(X,\vec{p})$ be a stable marked quasiprojective curve over $k$ of genus $g$ with $n$ marked points. The \emph{divisorial motivic zeta function of $(X,\vec{p})$} is defined to be
$$
Z_{\text{div}}\big(X,\vec{p};t\big) = \sum_{d \geq 0}\big[\Div^+_d(X,\vec{p})\big]t^d \ \ \ \ \ \ \in\ \  1 + t \cdot K_0(\text{Var} / k)\llbracket t\rrbracket.
$$
\end{definition}

 We begin by breaking up the classes $\big[\Div^+_d(X,\vec{p})\big]$ along their strata using the following lemma. 

\begin{lemma}[{{\cite[Chapter 2, Lemma 1.3.3]{ChambertLoirNicaiseSebag}}}]
\label{lemma:basic}
Suppose that $Y$ is a variety over $k$ and we have a decomposition $Y = Y_1 \sqcup \cdots \sqcup Y_r $ where all $Y_i$ are locally closed subvarieties of $Y$. Then:
$$\big[Y\big] = \big[Y_1\big] + \cdots + \big[Y_r\big].$$
\end{lemma}

\begin{lemma}
\label{lemma:sum}
Let $(X,\vec{p})$ be a stable marked quasiprojective curve over $k$. 
$$
\zdiv\big(X,\vec{p};t\big) 
     = \sum_{d \geq 0} t^d \sum_{\Delta(G,d)}\big[\Div_{(G',D)}(X,\vec{p})\big],
$$
where the second sum is over all stable pairs $(G',D)$ of degree $d$ over the dual graph $G$ of $(X, \vec{p})$. 
\end{lemma}
\begin{proof}
This follows from the description of the strata of $\Div^+_d(X,\vec{p})$ and Lemma \ref{lemma:basic}.
\end{proof}

\begin{notation}
Given a stable marked quasiprojective curve $(X,\vec{p})$ and points ${q_1, \ldots, q_{m} \in X}$, write 
$$
\big(X,\vec{p},-\vec{q}\big) := \big(X\backslash \{q_1, \ldots, q_{m}\},\{p_1,\ldots,p_n\} \backslash \{q_1, \ldots, q_{m}\}\big) .
$$
Given a connected component $X_v$ of $X$, we denote the non-special locus of $X_v$ by 
$$X_v^{(0)}:=\big(X_v,-\vec{p},-\text{Sing}(X_v)\big).$$
\end{notation}

We now describe the class of $\Div_{(G',D)}^+(X,\vec{p})$ in the Grothendieck ring. 

\begin{lemma}
\label{lemma:strata}
Let $(X,\vec{p})$ be a stable marked quasiprojective curve, and let $(G',D)$ be a stable pair such that the stabilization of $G'$ is equal to the dual graph $G$ of $X$. Then:
$$
\big[\Div_{(G',D)}(X,\vec{p})\big] = \prod_{v \in G}\big[(X_v^{(0)})_{D(v)}\big] \prod_{v' \in G' \backslash G}\big[\G_{D(v')-1}\big],
$$
where $\G$ denotes the one-dimensional algebraic torus $\A^1-\{0\}$ over $k$.
\end{lemma}
\begin{proof}
A point in $\Div_{(G',D)}^+(X,\vec{p})$ gives a divisor of degree $D(v)$ on $X_v^{(0)}$ for each $v \in G'$. On non-exceptional components, these are points in $(X_v^{(0)})_{D(v)}$. On exceptional components, these are points in~$\G_{D(v')-1}$. 
\end{proof}

We now prove a series of propositions which give us a way to iteratively relate the divisorial motivic  zeta function of a stable curve to the divisorial motivic zeta functions of its components. 


\begin{proposition}[Self-intersection]
\label{prop:loop}
Let $(X,\vec{p},q_1,q_2)$ be a stable quasiprojective curve with $n+2$ marked points and let $(X/_{q_1\sim q_2},\vec{p})$ be the curve with a nodal self-intersection obtained by gluing $q_1$ and $q_2$. Then:
$$
\zdiv\big(X,\vec{p},q_1,-q_2;t\big) = \zdiv\big(X/_{q_1\sim q_2},\vec{p};t\big).
$$
\end{proposition}
\begin{proof}
Using Lemmas \ref{lemma:sum} and \ref{lemma:strata}, we have that 
$$
    \zdiv\big(X/_{q_1\sim q_2},\vec{p};t\big) = \sum_{d \geq 0} t^d  \sum_{j = 0}^{d} \big[\Div_j^+(X,\vec{p},-p,-q)\big]\sum_{\alpha \in \text{comp}(d-j)}\big[\G\big]_\alpha
    = \zdiv\big(X,\vec{p},q_1,-q_2;t\big),
$$
where we interpret $j$ to be the degree of the divisor restricted to $(X,\vec{p},-q_1,-q_2)$. The last sum is taken over all ordered ways to write the integer $d -j $ as a sum of positive integers, and $\big[\G\big]_\alpha := \big[\G_{\alpha_1-1}\big]\cdots \big[\G_{\alpha_l-1}\big]$. We see this because in both cases the strata are given by stable pairs whose exceptional components emanate from the point $q_1$ (see Figure \ref{fig:loop}).
\end{proof}

\begin{figure}
\centering
\begin{minipage}{.5\textwidth}
  \centering
  \includegraphics[height= 1.6in]{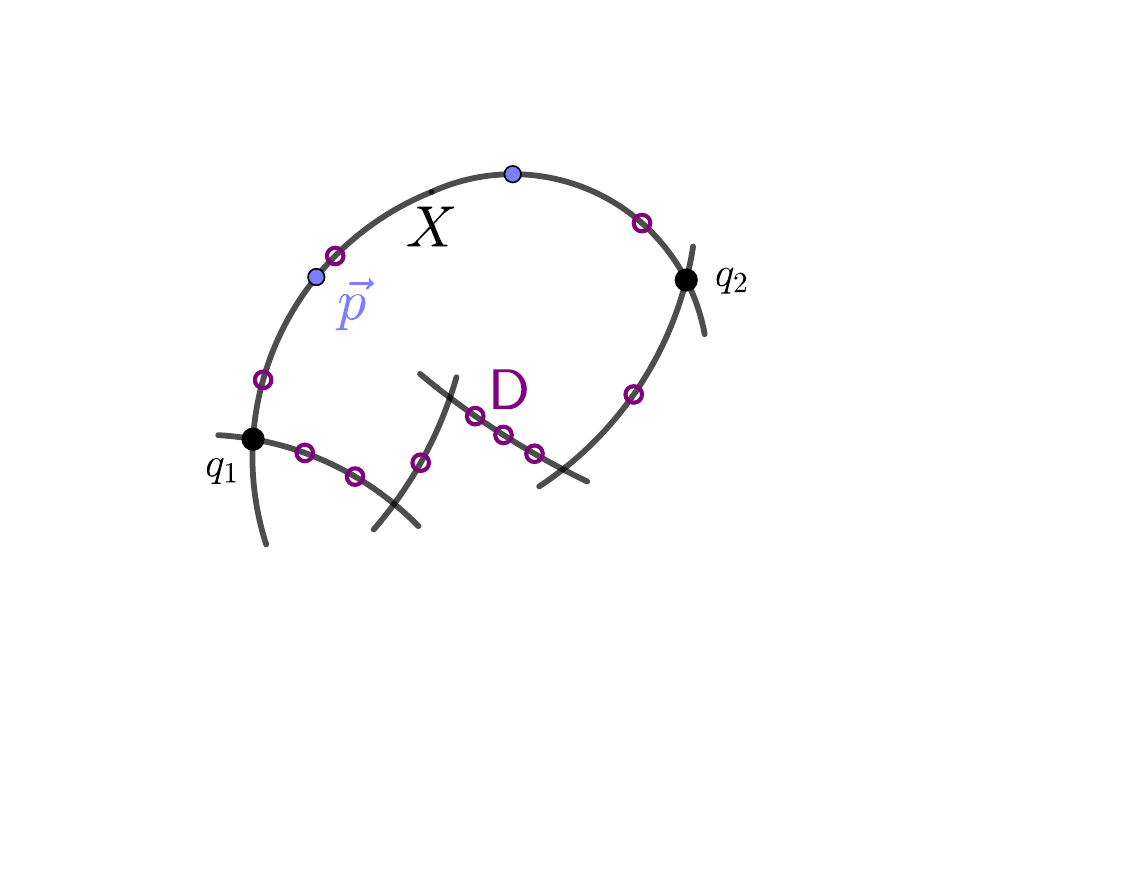}
  \captionof{figure}{A stable pair for Proposition \ref{prop:loop}}
  \label{fig:loop}
\end{minipage}%
\begin{minipage}{.5\textwidth}
  \centering
  \includegraphics[height= 1.6in]{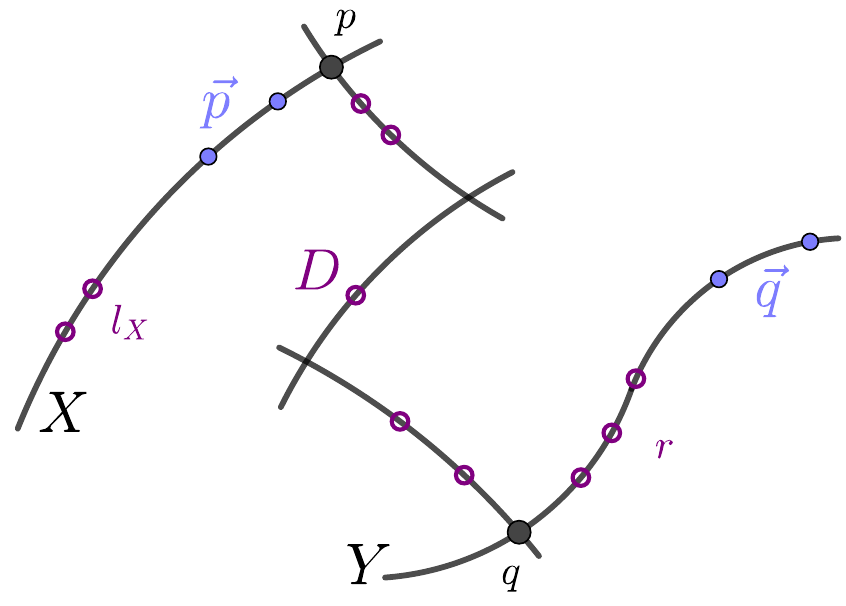}
  \captionof{figure}{A stable pair for Proposition \ref{prop:cross}}
  \label{fig:cross}
\end{minipage}
\end{figure}

\begin{proposition}[Intersection]
\label{prop:cross}
Let $(X,\vec{p},p)$ be a stable quasiprojective curve with $n+1$ marked points and let $(Y,\vec{q},q)$ be a stable quasiprojective curve with $m+1$ marked points. Denote by $X\sqcup_{p\sim q}Y$ the curve obtained by gluing $X$ to $Y$ along a node at the points $p,q$. Then:
$$
\zdiv\big(X\sqcup_{p\sim q}Y,\vec{p},\vec{q};t\big) = \zdiv\big(X,\vec{p},p;t\big)\cdot \zdiv\big(Y,\vec{q},-q;t\big).
$$
\end{proposition}
\begin{proof}
By Lemmas \ref{lemma:sum} and \ref{lemma:strata}, we have that the divisorial zeta function is
$$
\zdiv\big(X\sqcup_{p\sim q}Y,\vec{p},\vec{q};t\big)= \sum_{d \geq 0} t^d \sum_{r + l = d}\big[\Div_r^+(Y,\vec{q},-q)\big]
\left(
\sum_{l_x = 0}^l \big[\Div_{l_x}^+(X,\vec{p},-p)\big] \sum_{\alpha \in \text{comp}(l-l_x)}\big[\G\big]_\alpha
\right)
$$
where we interpret $r$ to be the degree of the divisor restricted to $Y$, $l$ to be the degree of the divisor restricted to $X$ and the exceptional components, and $l_x$ to be the degree of the divisor restricted to $X$ (See Figure \ref{fig:cross}). We observe that this is in fact a product of series:
$$
\zdiv\big(X\sqcup_{p\sim q}Y,\vec{p},\vec{q};t\big) =
\left(
\sum_{d \geq 0} t^d\big[\Div_d^+(Y,\vec{q},-q)\big]
\right)
\left(
\sum_{d \geq 0} t^d \sum_{d_x = 0}^d \big[\Div_{d_x}^+(X,\vec{p},-p)\big] \sum_{\alpha \in \text{comp}(d-d_x)}\big[\G\big]_\alpha
\right)
$$
On the left we have $\zdiv\big(Y,\vec{q},-q;t\big)$ and on the right we have  $\zdiv\big(X, \vec{p}, p;t\big)$.\
\end{proof}

\begin{proposition}[Closing]
\label{prop:closing}
Let $X$ be a smooth quasiprojective curve and let $p\in X$ be a point. Then,
$$
\zdiv\big(X,-p;t\big) = \zdiv(X;t)\cdot (1-t).
$$
\end{proposition}
First, we need the following lemma.
\begin{lemma}[{\cite[Chapter 7, Proposition 1.1.7]{ChambertLoirNicaiseSebag}}]\label{lemma:symmetricpower}
If $X$ is a quasiprojective variety, and $Y \hookrightarrow X$ is a closed subvariety with complement $U$, then
\begin{equation*}
   \big[X_n\big] = \sum_{i + j = n} \big[Y_i\big]\cdot\big[U_j\big].  
\end{equation*}
\end{lemma}
\begin{proof} [Proof of Proposition \ref{prop:closing}]
In our case, we take $Y = p$ and $U = X(-p)$. Then, 
$$\big[X_n\big] = \sum_{i+j=n}\big[1\big]\cdot\big[(X,-p)_j\big] = \sum_{j=0}^n \big[(X,-p)_j\big].$$
Since $X$ is smooth, we have $X_d = \Div_d^+(X)$.
Applying the above equation, we find
$$
    \zdiv(X;t) 
            = \sum_{d = 0}^\infty t^d  \sum_{j=0}^d \big[X(-p)_j\big] 
            = \zdiv(X,-p;t) (1 + t + t^2 + t^3 + \cdots) 
            = \frac{\zdiv(X,-p;t)}{1-t}.
$$
\end{proof}

\begin{proposition} Denote by $\G$ be the one-dimensional algebraic torus $\A^1-\{0\}$ over $k$ and by $\LL$ the class of $\A^1$ in the Grothendieck ring. Then:
\label{prop:gm}
$$
\zdiv(\G;t) = \zmot(\G;t)=\frac{1-t}{1-\LL t}.
$$
\end{proposition}

\begin{proof}
By Lemma \ref{lemma:symmetricpower} applied to $X = \mathbb{A}^1$ and $U = \G$, we find that $\big[\A^1_d\big] = \sum_{i = 0}^d \big[\G_i\big]$. Using the fact that $(\A^1)_d = \A^d$, and therefore $\big[(\A^1)_d\big]=\LL^d$, we have:
\begin{equation*}
        \zdiv(\G;t)   
    = \sum_{d \geq 0} t^d \big[\G_d\big]
    = \sum_{d \geq 0} t^d \left(\LL^d - \sum_{i = 0}^{d-1} \big[\G_{i}\big]\right) 
            = \sum_{d \geq 0} t^d \LL^d  - \sum_{d \geq 0} t^d \sum_{i=0}^{d-1}\big[\G_{i}\big].
\end{equation*}
After re-indexing, we find:
\begin{equation*}
\begin{split}
    \zdiv(\G;t)   
            &= \frac{1}{1-\LL t} - t \sum_{d \geq 0} t^d \sum_{i = 0}^d\big[\G_{i}\big]\\
           &=\frac{1}{1-\LL t} - t \cdot \zdiv(\G;t)\cdot (1 + t + t^2 + \cdots)
           = \frac{1}{1-\LL t} - \frac{t \cdot\zdiv(\G;t) }{1-t} .
\end{split}
\end{equation*}
Solving for $\zdiv(\G,t)$, we find that $\zdiv(\G,t) = \frac{1-t}{1-\LL t}$, as claimed. The equality $\zdiv(\G;t) = \zmot(\G;t)$ holds, since $\GG$ is a smooth curve without marked points.
\end{proof}

\begin{proposition}
\label{prop:marked}
Let $(X, \vec{p},q)$ be a quasiprojective stable marked curve. Then,
$$
\zdiv\big(X, \vec{p}, q;t\big) = \zdiv\big(X, \vec{p},-q;t\big) \cdot \frac{1- \LL t}{1 - \mathbb{L}t - t +t^2}.
$$
\end{proposition}

\begin{proof} We have
\begin{equation*}
\begin{split}
    \zdiv\big(X,\vec{p},q;t\big)    &= \sum_{d \geq 0} t^d \sum_{j = 0}^d \big[\Div_{d-j}^+(X, \vec{p},-q)\big] \sum_{\alpha \in \text{comp}(j)} \big[\G\big]_\alpha \\
                &= \left(\sum_{d\geq 0} \big[\Div_d^+(X,\vec{p},-q)\big]t^d\right) \left(1+ \sum_{d \geq 1} t^d \sum_{\alpha \in \text{comp}(d)} \big[\G\big]_\alpha \right),
\end{split}
\end{equation*}
where we think of $j$ as the degree of the divisor restricted to exceptional components. 
Now, we evaluate the right term in this product. Re-organizing by the length of the composition, we find
\begin{equation*}\begin{split}
    1+\sum_{d \geq 1} t^d \sum_{\alpha \in \text{comp}(d)} \big[\G\big]_\alpha 
    &= 1+ \sum_{k\geq 1}\sum_{d\geq k}t^d\sum_{\substack{{\alpha\in\text{comp}(d)}\\ {\vert\alpha\vert=k}}} \big[\G\big]_\alpha\\
    &= 1+\sum_{k \geq 1} \big(t \cdot \zdiv(\G;t)\big)^k
    = \frac{1}{1-t \cdot \zdiv(\G,t)} .
\end{split}\end{equation*}
Applying Proposition \ref{prop:gm} and simplifying, we obtain the result.
\end{proof}

We are now ready to prove Theorem \ref{thm:main} from the introduction.
\begin{proof}[Proof of Theorem \ref{thm:main}]
Let $(X,\vec{p})$ be a pointed stable curve over $k$ with dual graph $G$. 
We use Propositions \ref{prop:loop} and \ref{prop:cross} to break up $X$ into its components.
Each node in $X$ yields a new marked point and a new hole. By Proposition \ref{prop:marked}, exchanging the $|E| + n$ marked points for a holes leads to factors of 
$$
\frac{1- \LL t}{1 - \LL t - t + t^2},
$$
Stitching the $2|E| + n$ holes leads to factors of $1-t$ by Proposition~\ref{prop:closing}. So we obtain
$$
\zdiv(X, \vec{p};t) = \left(\frac{1-\LL t}{1-\LL t-t+t^2}\right)^{|E|+n} (1-t)^{2 |E|+n} \prod_{v \in V}  \text{Z}_{\text{div}}\big(\widetilde{X}_v,t\big).
$$
Finally, we use that the motivic zeta function is equal to the divisorial zeta function for each $\widetilde{X}_v$ because $\widetilde{X}_v$ is smooth and does not have marked points.
\end{proof}


\bibliographystyle{amsalpha}
\bibliography{sample}{}

\end{document}